\documentclass[letterpaper, 10 pt, conference]{ieeeconf}
\IEEEoverridecommandlockouts                              % This command is only needed if 
\usepackage{graphicx}
\usepackage{mathrsfs,amsmath}
\usepackage[hidelinks]{hyperref}
\usepackage{amssymb}
\usepackage{amsfonts}
\usepackage{mathtools}
\usepackage{caption}
\usepackage{subcaption}
\usepackage{algorithm}
\usepackage{algorithmic}
\usepackage{float}
\usepackage{xcolor}

\newtheorem{definition}{Definition}
\newtheorem{theorem}{Theorem}

\newtheorem{lemma}{Lemma}

\newtheorem{assumption}{Assumption}

\usepackage{cite}

\newcommand{\TODO}[1]{{\color{black}#1}}
\begin{document}

\title{Constant-Sum High-Order Barrier Functions\\
for Safety Between Parallel Boundaries}

\author{Kwang Hak Kim, Mamadou Diagne, and Miroslav Krstić% <-this % stops a space
\thanks{This work was supported by the Department of the Navy, Office of Naval Research under grant N00014-23-1-2376. The results and opinions in this paper are solely of the authors and do not reflect the position or the policy of the U.S. Government.\endgraf
K. Kim, M. Diagne, and M. Krstić are with the Department of Mechanical and Aerospace Engineering, UC San Diego, 9500 Gilman Drive, La Jolla, CA, 92093-0411, {\tt\small \{kwk001,mdiagne,krstic\}@ucsd.edu}.}
}

\maketitle
% \thispagestyle{empty}
% \pagestyle{empty}
%%%%%%%%%%%%%%%%%%%%%%%%%%%%%%%%%%%%%%%%%%%%%%%%%%%%%%%%
\begin{abstract}
% Control Barrier Functions (CBFs) are widely used in safety-critical control for their useful properties, yet their construction and validation continue to be a challenge, especially for high relative degree CBFs. 
This paper takes a step towards addressing the difficulty of constructing Control Barrier Functions (CBFs) for parallel safety boundaries. A single CBF for both boundaries has been reported to be difficult to validate for safety, and we identify why this challenge is inherent. To overcome this, the proposed method constructs separate CBFs for each boundary. We begin by presenting results for the relative degree one case and then extend these to higher relative degrees using the CBF backstepping technique, establishing conditions that guarantee safety. Finally, we showcase our method by applying it to a unicycle system, deriving a simple, verifiable condition to validate the target CBFs for direct implementation of our results.
\end{abstract}

% \begin{IEEEkeywords}
% Autonomous systems; Constrained control; Backstepping; Autonomous vehicles
% \end{IEEEkeywords}

%%%%%%%%%%%%%%%%%%%%%%%%%%%%%%%%%%%%%%%%%%%%%%%%

\section{Introduction}
Safety is a growing concern as autonomous technologies expand across various sectors, increasing interest in safety-critical control. Control Barrier Functions (CBFs) \cite{ames2019cbf_theory,wieland_constructive_2007} have gained popularity for their ability to enforce safety, with minimal deviation from nominal objectives when used with the quadratic program (QP) formulation \cite{ames2014cbfqp_cruise,ames2017cbfqp_critical} leading to many findings \cite{abel2023prescribed_time,velimir_safe,koga_safe_2023}. However, implementing CBF-QP safety filters can be challenging, particularly for high and mixed relative degree CBFs, where absent control inputs in the first derivative of the safety constraint cannot give safety guarantees. To address this, several methods have been proposed, with High Order CBFs (HOCBFs) \cite{xiao_hocbf} being among the most widely used. \TODO{Nevertheless, constructing such CBFs remains a nontrivial task.

We focus on the case where the objective is to keep the system within parallel boundaries. In this scenario, the construction of a valid CBF can be particularly challenging at the midway point between the boundaries, where the gradient of the safety constraint often vanishes. This leads to a loss of control authority, making safety enforcement difficult in any direction. Such issues were observed in Examples 2 and 5 of \cite{cohen_backstep_rom_2024}, where traditional single CBF approaches with high relative degree techniques failed to ensure safety. While methods such as \cite{taylor_safe_2022} and \cite{ong_rectified_2024} have been proposed, we take a different approach, drawing insights from \cite{drew_brugg_sim_LK} and \cite{koga_safe_2023}, and generalize them into a more flexible and systematic framework.

In this paper, we show that utilizing two CBFs for parallel boundaries can provide advantages over using a single CBF. Although numerous studies have explored the use of multiple CBFs, such as \cite{molnar2023composing} and \cite{alyaseen2024safetycriticalcontroldiscontinuoussystems}, this work uses the control-sharing property as presented in \cite{XU2018195}. However, the work in \cite{XU2018195} assumes that the Lie derivative in the control direction is never zero. Our approach relaxes this assumption by explicitly verifying feasibility when the Lie derivative is zero, thereby accommodating potential sign changes of the Lie derivative. Furthermore, we extend the single-input result in \cite{XU2018195} to multi-input systems when the safety boundary is parallel. 

This formulation is extended to high relative degree CBFs by utilizing the CBF backstepping technique, having roots in \cite{krstic_nonovershooting_2006} that predates the formulation of CBFs. Modern translations to CBFs can be found in \cite{abel2023prescribed_time,koga_safe_2023,kim2024robustcontrolbarrierfunction}. While the target CBF transformations in this technique are similar to HOCBFs, the key difference is in the gain design. By choosing gains based on the system's initial condition, the induced safe set guarantees the initial states are included, eliminating the need to retune parameters as in HOCBFs. Due to this useful property, the backstepping method has been applied to works such as simultaneous lane-keeping and obstacle avoidance \cite{drew_brugg_sim_LK}, prescribed time safety filters (PTSf) \cite{abel2023prescribed_time}, Stefan Model PDEs \cite{koga_safe_2023,koga2023event}, and unknown obstacle avoidance \cite{kim2024robustcontrolbarrierfunction}.

We note that the above CBF backstepping method differs from the backstepping approach defined in \cite{taylor_safe_2022}, which also addresses high relative degree problems but is inspired by traditional Lyapunov backstepping, where virtual controls transform a system into a target system. While both methods ask the same fundamental question, they differ significantly in the choice of the target system. For clarity, when we say CBF backstepping, we refer to the backstepping as in \cite{krstic_nonovershooting_2006}.

This work proposes a framework using two CBFs to enforce parallel safety boundaries, applicable to both relative degree one and arbitrarily high relative degree CBFs with the associated QP admitting a closed-form solution. A unicycle example illustrates the method’s effectiveness, including a sufficient condition for the target CBF's validity.
}

\section{Preliminaries} \label{prelim}

We start with a preliminary overview of CBFs \cite{ames2019cbf_theory} outlining their definitions and key properties. Consider the following nonlinear control affine system:
\begin{equation} \label{affine}
    \dot{x} = f(x) + g(x)u,
\end{equation}
with $x \in D \subseteq\mathbb{R}^{n}$ and \TODO{unbounded input} $u \in \mathbb{R}^{m}$.

\begin{definition}
    A scalar-valued continuously differentiable function $h \ : \ D \rightarrow \mathbb{R}$ with the property that $\inf_{x \in D} h(x) < 0$ and $\sup_{x \in D} h(x) > 0$ is defined as a \textit{Barrier Function (BF) candidate}. The set $\mathcal{C} = \{ x \in D \mid h(x) \geq 0\}$ is defined as the \textit{safe set}, the set $\text{Int}(\mathcal{C}) \; = \; \{x \in D \mid h(x) > 0 \}$ is the \textit{interior set}, and the set $\partial \mathcal{C} = \{ x \in D \mid h(x) = 0\}$ is the \textit{boundary set}.
\end{definition}
% \begin{definition}
%     (Forward Invariant/Safe): The set $\mathcal{C}$ is forward invariant if $\forall x_{0} \in \mathcal{C}, x(t) \in \mathcal{C}$ for $x(0) = x_{0}$ and all $t \in I(x_0)$. Then, the system
%         \begin{equation} \label{lipschitz_affine}
%             \dot{x} = f_{\text{cl}}(x) \coloneqq f(x) + g(x)k(x)
%         \end{equation}
%     is safe w.r.t the set $\mathcal{C}$ if the set $\mathcal{C}$ is forward invariant.
% \end{definition}

\begin{definition}
    A continuously differentiable function $h \; : \; D \rightarrow \mathbb{R}$, is a \textit{Control Barrier Function (CBF)} if there exists an extended class $\mathcal{K}_{\infty}$ function $\alpha \; : \; \mathbb{R} \rightarrow \mathbb{R}$ such that for the control system \eqref{affine}:
    \begin{equation}
        \sup_{u} [L_{f}h(x) + L_{g}h(x)u] \geq -\alpha (h(x)).\label{cbf_cond}
    \end{equation}
\end{definition}

Note that since \TODO{the control input is assumed unbounded (i.e., $u \in \mathbb{R}^m$)}, condition \eqref{cbf_cond} is equivalent to:
\begin{align}
    L_gh(x) = 0 \implies L_fh(x) \geq -\alpha(h(x)).\label{cbf_Lgh_cond}
\end{align}
\begin{lemma}\label{lemma_forward_invariance}
    \textit{(Lemma 4.4 in \cite{khalil_nonlinear_2002} and Lemma 2 in \cite{glotfelter2017nonsmooth})} For a continuously differentiable function $h:\; [t_0,\infty) \rightarrow \mathbb{R}$, if $h(x(t_0)) \geq 0$ and $\dot{h}(x(t)) \geq -\alpha(h(x(t)))$ for all $t \in [t_0,\infty)$, then $h(x(t)) \geq 0$ for all $t \in [t_0,\infty)$.
\end{lemma}

We now introduce our motivational problem.

\section{Motivational Problem}\label{motivation}
The following problem was highlighted in \cite{cohen_backstep_rom_2024} and \cite{ong_rectified_2024}. Consider the double integrator system: $\dot{x}_1 = x_2, \; \dot{x}_2 = u$, where ${x = [x_1, x_2]^\top \in \mathbb{R}^2}$. To enforce ${x_1 \in [-1,1]}$, we define the safe set as ${ \mathcal{C} = \{x \in \mathbb{R}^2 \mid x_1 \in [-1,1]\}}$. A common CBF choice is $h_1(x) = 1 - x_1^2$, but since $L_g h_1(x) = 0$ for all $x$, it can be seen to have a relative degree of 2. 

To address this, we use the CBF backstepping method to construct a target CBF ${h_2(x) = c_1(1 - x_1^2) - 2x_1x_2}$, with the gain choice \TODO{$c_1 > 0$. 
% appropriate gains $c_1, c_2 > 0$. 
The resulting CBF condition is $-2c_1x_1x_2 -2x_2^2 -2x_1u \geq -c_1c_2(1 - x_1^2) -2c_2x_1x_2,$
% \begin{align}
%     -2c_1x_1x_2 -2x_2^2 -2x_1u \geq -c_1c_2(1 - x_1^2) -2c_2x_1x_2,\label{doubint_safety}
% \end{align}
where $c_2 > 0$ is the liveness gain for $h_2(x)$}.
We observe that at $x_1 = 0$, the safety filter loses control authority (i.e., $L_g h_2(x) = L_g L_f h_1(x) = -2x_1 = 0$). When $x_1 = 0$, the CBF condition requires that $|x_2| \leq \sqrt{c_1 c_2 / 2}$, which does not hold for all $x_2 \in \mathcal{C}$ and therefore cannot guarantee safety. Similar conditions arise when using HOCBFs as reported in \cite{cohen_backstep_rom_2024} and \cite{ong_rectified_2024}. 

This problem arises because the gradient $\nabla h_1(x) = [-2x_1, 0]^\top$ vanishes at $x_1 = 0$, where $h_1(x)$ has a local maximum. Instead, we propose an alternative approach that decomposes the single CBF constraint into two, removing the loss of control authority due to the gradient vanishing while maintaining safety. Our method generalizes ideas from \cite{drew_brugg_sim_LK} and \cite{koga_safe_2023} in a more flexible framework.

\section{Control Barrier Functions with Parallel Safety Boundaries}\label{parallel}
Consider two CBFs, $h(x)$ and $\hbar(x)$, for system \eqref{affine}, both having uniform relative degree $r = n \geq 1$, with corresponding safe sets $\mathcal{C}_h$ and $\mathcal{C}_\hbar$. The ideal safe set is then given by the intersection of these sets, i.e., $\mathcal{C} = \mathcal{C}_h \cap \mathcal{C}_{\hbar}$. We introduce a class of CBFs in the form of the following definition, which makes this process streamlined. 

\begin{definition}
    A pair of candidate CBFs $h(x)$ and $\hbar(x)$ is \textit{parallel} if there exists a constant $b > 0$ such that:
    \begin{align}
        h(x) + \hbar(x) = b\label{parallel_cond}.
    \end{align}
\end{definition}

We note that the constant-sum expression in \eqref{parallel_cond} is a special case of the convex combination $a_1 h(x) + a_2 \hbar(x) = b$ with $a_1, a_2 > 0$. However, since the scaling factors $a_1, a_2$ can be rescaled into the CBFs without affecting the safe set, we opt to study the constant-sum formulation to clearly illustrate its utility.

The parallel structure is particularly useful as it allows for a more systematic expression of the combined safety condition. A key property of the parallel CBFs is that $\nabla h(x) = -\nabla \hbar(x)$, which represents two parallel safety boundaries with the combined safe set defined between them. We now demonstrate how this property can be used to derive a safety filter. 

\subsection{Relative Degree One}
We first present the result for the case of relative degree one. Suppose $h(x)$ and $\hbar(x)$ are parallel and both have a relative degree of $r = 1$. Taking the time derivative of $h(x)$ and $\hbar(x)$ yields the CBF conditions:
\begin{align}
    L_fh(x) + L_gh(x)u \geq -\alpha(h(x)),\label{rel1_hdot}\\
    L_f\hbar(x) + L_g\hbar(x)u \geq -\overline{\alpha}(\hbar(x))\label{rel1_hbardot},
\end{align}
where $\alpha$ and $\overline{\alpha}$ are extended class $\mathcal{K}_\infty$ functions. Since $h(x)$ and $\hbar(x)$ are parallel, without loss of generality, \eqref{rel1_hbardot} can be rewritten as $-L_fh(x) - L_gh(x)u \geq -\overline{\alpha}(b-h(x))$ and combining with \eqref{rel1_hdot}, we arrive at the following joint constraint:
\begin{align}\label{constraint}
    % \underbrace{-L_fh(x) -\alpha(h(x))}_{\underline{\eta}(x)} 
    % \leq L_gh(x)u \leq 
    % \underbrace{-L_fh(x) + \overline{\alpha}(b-h(x))}_{\overline{\eta}(x)}
 &   {\underline{\eta}(x)} 
    \leq L_gh(x)u \leq
   {\overline{\eta}(x)},
    \\
    \label{constraint1}
   \underline{\eta}(x) &:= -L_fh(x) -\alpha(h(x)) ,
   \\
   \overline{\eta}(x) &:= -L_fh(x) + \overline{\alpha}(b-h(x)).
    \label{constraint2}
\end{align}

Then, to derive the control law, we introduce the quadratic program (QP)\cite{ames2014cbfqp_cruise} formulation as follows:
\begin{align}
    &\qquad \quad u = \min_{u} \frac{1}{2}\| u - u_0 \|^2, \label{qp_1}\\
    &\text{Subject to:} \quad  \underline{\eta}(x) \leq  L_gh(x) u \leq \overline{\eta}(x)\label{qp_2}.
\end{align}

First, we establish that the QP problem \eqref{qp_1}--\eqref{qp_2} is always feasible when $h(x)$ and $\hbar(x)$ are valid CBFs and are parallel. A necessary and sufficient condition for ensuring feasibility in such cases is to show that $h(x)$ and $\hbar(x)$ have the \textit{control-sharing property} \cite[Thm.~1]{XU2018195}. That is, feasibility is guaranteed if the inequality $\underline{\eta}(x) \leq \overline{\eta}(x)$ holds for all $x \in D$. However, \cite{XU2018195} assumes that $L_g h(x)$ and $L_g \hbar(x)$ are never zero. This was a restrictive assumption that also limits the sign of the $L_g h(x),L_g \hbar(x)$ terms to be constant. To address cases where this assumption does not hold, we further show that $L_g h(x) = -L_g \hbar(x) = 0 \implies \underline{\eta}(x) \leq 0 \leq \overline{\eta}(x)$. The following lemma formalizes this result.

\begin{lemma}\label{lemma_feasible}
    Let $h(x)$ and $\hbar(x)$ be valid CBFs and parallel. Then, ${L_g h(x) = 0 \implies \underline{\eta}(x) \leq 0 \leq \overline{\eta}(x)}$. Moreover, there exist extended class $\mathcal{K}_{\infty}$ functions $\alpha$ and $\overline{\alpha}$ such that $\underline{\eta}(x) \leq \overline{\eta}(x)$ for all $x \in D$. Thus, $h(x)$ and $\hbar(x)$ have the control-sharing property and the QP problem \eqref{qp_1}--\eqref{qp_2} is always feasible.
\end{lemma}
\begin{proof}
    Since \TODO{$h(x)$ and $\hbar(x)$} 
    % $h_n(x)$ and $\hbar_n(x)$ 
    are CBFs, from \eqref{cbf_Lgh_cond}, we have the following properties:
    \begin{align}
            L_gh(x) = 0 &\implies L_fh(x) \geq -\alpha(h(x)),\label{CBF_h_cond}\\ 
            L_g\hbar(x) = 0 &\implies L_f\hbar(x) \geq -\overline{\alpha}(\hbar(x)).\label{CBF_hbar_cond}
    \end{align}
    Then, using the fact that $h(x)$ and $\hbar(x)$ are parallel, we rewrite \eqref{CBF_hbar_cond} as:
    \begin{align}
        -L_gh(x) = 0 \implies -L_fh(x) \geq -\overline{\alpha}(b - h(x))\label{CBF_hbar_cond2}.
    \end{align}
    Since ${-L_gh(x) = 0 \iff L_gh(x) = 0}$, we see from \eqref{CBF_h_cond} and \eqref{CBF_hbar_cond2} that $L_gh(x) = 0 \implies \underline{\eta}(x) = -L_fh(x) - \alpha(h(x)) \leq 0$ and $\overline{\eta}(x) = -L_fh(x) + \overline{\alpha}(b - h(x)) \geq 0$.
    Thus, $L_gh(x) = 0 \implies \underline{\eta}(x) \leq 0 \leq \overline{\eta}(x)$.
    
    To prove the second part, we observe that $ \overline{\eta}(x) - \underline{\eta}(x) = \overline{\alpha}(b-h(x)) + \alpha(h(x))$. The trivial choice of $\overline{\alpha}(s) = \alpha(s) = cs$ where $c > 0$ results in $\overline{\eta}(x) - \underline{\eta}(x) = cb > 0$. Thus, there exist extended class $\mathcal{K}_{\infty}$ functions $\alpha$ and $\overline{\alpha}$ such that $\underline{\eta}(x) \leq \overline{\eta}(x)$ for all $x \in D$.
\end{proof}

With the feasibility established, we prove that the QP problem \eqref{qp_1}--\eqref{qp_2} has a closed-form solution.

\begin{theorem}\label{thrm_kkt}
    The QP problem \eqref{qp_1}--\eqref{qp_2} has a closed-form solution of the form:
    \TODO{
    \begin{align}\label{control_law_rel1}
        u^* = \begin{cases}
            u_0, & \text{if } L_g h(x) = 0, \\
            u_0, & \text{if } \underline{\eta}(x) \leq L_g h(x) u_0 \leq \overline{\eta}(x), \\
            u_0 + \overline{\psi}(x,u_0), & \text{if } L_g h(x) u_0 > \overline{\eta}(x), \\
            u_0 + \underline{\psi}(x,u_0), & \text{if } L_g h(x) u_0 < \underline{\eta}(x),
            \end{cases}
    \end{align}
    where 
    \begin{align}
        \overline{\psi}(x,u_0) \coloneqq \Bigl( L_g h(x)\Bigr)^\top\dfrac{\overline{\eta}(x) - L_g h(x) u_0 }{| L_g h(x) |^2},\\
        \underline{\psi}(x,u_0) \coloneqq \Bigl( L_g h(x)\Bigr)^\top\dfrac{\underline{\eta}(x) - L_g h(x) u_0 }{| L_g h(x) |^2}.
    \end{align}
    }
\end{theorem}
\begin{proof}
    In the case where $L_g h(x) = 0$, we have shown in Lemma \ref{lemma_feasible} that $L_g h(x) = 0 \implies \underline{\eta}(x) \leq 0 \leq \overline{\eta}(x)$, regardless of the control input $u$. Hence, the optimal solution to \eqref{qp_1} is $u^* = u_0$.  

    For the case $L_gh(x) \neq 0$, we must verify systematically. Noting that both the objective function and constraints are convex, the Karush-Kuhn-Tucker (KKT) conditions \cite{boyd2004convex} provide necessary and sufficient conditions for $u^* \in \mathbb{R}^m$ to be the optimal solution. Namely, we must show that there exist $\lambda_1, \lambda_2 \geq 0$ such that  
    \TODO{
    \begin{align}
       & u^* - u_0 + (\lambda_1 - \lambda_2)L_g h(x)^\top = 0, \label{stationary} \\
        &\underline{\eta}(x) \leq L_g h(x) u^* \leq \overline{\eta}(x), \label{primal} \\
        &\lambda_1, \lambda_2 \geq 0, \label{dual} \\
        &\lambda_1 (L_g h(x) u^* - \overline{\eta}(x)) = 0, \label{slackness1} \\
        &\lambda_2 (\underline{\eta}(x) - L_g h(x) u^*) = 0. \label{slackness2}
    \end{align}  
    }
    Here, \eqref{stationary} is referred to as the \textit{stationarity} condition, \eqref{primal} as the \textit{primal feasibility} condition, \eqref{dual} as the \textit{dual feasibility} condition, and \eqref{slackness1}--\eqref{slackness2} as the \textit{complementary slackness} condition. 
    
    \textbf{Case 1:} Suppose $\underline{\eta}(x) < L_gh(x)u^* < \overline{\eta}(x)$. Then, $\lambda_1 = \lambda_2 = 0$, as seen from \eqref{slackness1}--\eqref{slackness2} since neither constraint is active. From \eqref{stationary}, we get $u^* = u_0$, which matches the expression in \eqref{control_law_rel1}.
    
    %$P(x,u_0) = L_gh(x)u_0$ and simplifies to $u^* = u_0$.
    \textbf{Case 2:} Suppose $L_gh(x)u^* = \overline{\eta}(x)$. Then, $\lambda_2 = 0$ and from \eqref{stationary} we get the expression $u^* = u_0- \lambda_1L_gh(x)^\top$ to get that $L_gh(x)u_0 - \lambda_1|L_gh(x)|^2 = \overline{\eta}(x)$ from \eqref{slackness1}. Then, solving for $\lambda_1$:
    \begin{align}
        \lambda_1 = -\frac{\overline{\eta}(x)-L_gh(x)u_0}{|L_gh(x)|^2}.
    \end{align}
    From \eqref{dual}, since $\lambda_1 \geq 0$, it follows that if $L_g h(x) u^* = \overline{\eta}(x)$, then we must have $L_g h(x) u_0 > \overline{\eta}(x)$. Substituting the expression for $\lambda_1$ into the expression for $u^*$ again, we obtain  
    \begin{align}
        u^* = u_0 + \Bigl(L_g h(x)\Bigr)^\top \frac{\overline{\eta}(x) - L_g h(x) u_0}{|L_g h(x)|^2},
    \end{align}  
    which matches with \eqref{control_law_rel1} when $L_g h(x) u_0 > \overline{\eta}(x)$.
    % and $L_gh(x) \neq 0$.  
    
    \textbf{Case 3:} Similar to case 2.

    Thus, the control law \eqref{control_law_rel1} is a closed-form solution to the QP problem \eqref{qp_1}--\eqref{qp_2}.
\end{proof}

We summarize this result in Theorem \ref{thrm_rel1_Safe}.
\begin{theorem}\label{thrm_rel1_Safe}
    For system \eqref{affine}, if $h(x)$ and $\hbar(x)$ are CBFs and parallel, then there exist extended class $\mathcal{K}_\infty$ functions $\alpha$ and $\overline{\alpha}$ such that the control law \eqref{control_law_rel1} keeps the set $\mathcal{C}_h \cap \mathcal{C}_\hbar$ forward invariant.
\end{theorem}
\begin{proof}
    Lemma \ref{lemma_feasible} establishes the existence of extended class $\mathcal{K}_\infty$ functions $\alpha$ and $\overline{\alpha}$ such that the QP problem \eqref{qp_1}--\eqref{qp_2} is always feasible. Then, the closed-form solution given in Theorem \ref{thrm_kkt} guarantees that $\dot{h}(x(t)) \geq -\alpha(h(x(t)))$ and $ \dot{\hbar}(x(t)) \geq -\overline{\alpha}(\hbar(x(t)))$ for all $t \in [t_0, \infty)$. If the initial conditions satisfy $h(x(t_0)) \geq 0$ and $ \hbar(x(t_0)) \geq 0$, by Lemma \ref{lemma_forward_invariance}, it follows that $h(x(t)) \geq 0$ and $\hbar(x(t)) \geq 0$ for all $t \in [t_0, \infty)$. Hence, the set $\mathcal{C}_h \cap \mathcal{C}_\hbar$ is forward invariant.
\end{proof}

\subsection{High Relative Degree}

Suppose $h(x)$ and $\hbar(x)$ are parallel with relative degree $r = n > 1$. 
% In this case, additional techniques are needed to enforce the safety constraint, as $L_g h(x)$ and $L_g \hbar(x)$ are both zero. 
We apply the CBF backstepping technique, starting with the following necessary assumption.
% Now, suppose $h(x)$ and $\hbar(x)$ are parallel but have relative degree $r = n > 1$. In this case, additional techniques are required to construct a safety filter that enforces the safety constraint, as $L_g h(x)$ and $L_g \hbar(x)$ are identically zero. Hence, we use the CBF backstepping technique to approach this problem. We first make the following assumption that is necessary for the backstepping technique.

\begin{assumption}\label{assum_backstepping}
    $h(x)$ and $\hbar(x)$ are $n$-times differentiable with  $h(x_0)>0$ and $\hbar(x_0)>0$. That is, $x_0 \in \text{Int}(\mathcal{C}_h)\cap\text{Int}(\mathcal{C}_{\hbar})$.
\end{assumption}

% \begin{assumption}\label{assum_interior}
%     $h(x_0)>0$ and $\hbar(x_0)>0$. That is, $x_0 \in \text{Int}(\mathcal{C}_h)\cap\text{Int}(\mathcal{C}_{\hbar})$.
% \end{assumption}

% We note that Assumption \ref{assum_interior} imposes a restriction only on the initial condition required for the gain design in the CBF backstepping algorithm. It does not restrict the boundary of $\mathcal{C}_h\cap\mathcal{C}_\hbar$ from being reached after the initial state. Additionally, we assume that $h_1(x)$ and $\hbar_1(x)$ are $n$-times differentiable.

% \begin{assumption}\label{assum_differentiable}
%     $h_1(x)$ and $\hbar_1(x)$ are $n$-times differentiable.
% \end{assumption}

Now, we define $h_1(x) \coloneqq h(x)$, $\hbar_1(x) \coloneqq \hbar(x)$, and $b_1 \coloneqq b$ to proceed with the CBF backstepping algorithm. Taking the time derivative of $h_1(x)$ and $\hbar_1(x)$, and adding zero yields:
\begin{align}
    \dot{h}_1 (x) &= -c_1 h_1(x) + \underbrace{c_1 h_1(x) + L_fh_1(x)}_{h_2},\\
    \dot{\hbar}_1 (x) &= -\bar{c}_1 \hbar_1(x) + \bar{c}_1 \hbar_1(x) + L_f\hbar_1(x),\\
    &= -\bar{c}_1 \hbar_1(x)  + \underbrace{\bar{c}_1 (b_1-h_1(x)) - L_fh_1(x)}_{\hbar_2},
\end{align}
% \begin{align}
%     \dot{h}_1 (x) &= L_fh_1(x),\\
%     &= -c_1 h_1(x) + \underbrace{c_1 h_1(x) + L_fh_1(x)}_{h_2},\\
%     \dot{\hbar}_1 (x) &=- L_fh_1(x),\\
%     &= -\bar{c}_1 \hbar_1(x) + \bar{c}_1 \hbar_1(x) - L_fh_1(x),\\
%     &= -\bar{c}_1 \hbar_1(x)  + \underbrace{\bar{c}_1 (b_1-h_1(x)) - L_fh_1(x)}_{\hbar_2},
% \end{align}
which gives:
\begin{align}
    h_2(x) &\coloneqq c_1 h_1(x) + L_fh_1(x),\\
    \hbar_2(x) &\coloneqq \bar{c}_1 (b_1-h_1(x)) - L_fh_1(x)\label{hbar_2}.
\end{align}

Then, from Assumption \ref{assum_backstepping}, we can choose the gains $c_1,\bar{c}_1 > 0$ as:
\begin{align}
    c_1 = \bar{c}_1 > \max\left\{ -\frac{L_fh_1(x_0)}{h_1(x_0)}, \frac{L_fh_1(x_0)}{b_1 - h_1(x_0)}\right\} \geq 0,
\end{align}
to guarantee that $h_2(x_0), \hbar_2(x_0) > 0$. This also allows us to rewrite \eqref{hbar_2} as $\hbar_2(x) = b_2 - h_2(x)$, where we define $b_2 \coloneqq c_1b_1 > 0$.

We now iterate this procedure until the target CBF transformation is acquired.
\begin{align}
    h_1(x) &= h(x),\label{cbf_chain_start}\\
    \hbar_1(x) &= b_1 - h(x),\\
    h_i(x) &= c_{i-1}h_{i-1}(x) + L_fh_{i-1}(x),\label{hi}\\
    \hbar_i(x) &= b_i - h_{i}(x),\label{hbari}
\end{align}
for $i = \{2, \cdots, n\}$ where $b_i \coloneqq c_{i-1}b_{i - 1}$ and
\begin{align}
    c_{i-1} &> \max\left\{ -\frac{L_fh_{i-1}(x_0)}{h_{i-1}(x_0)}, \frac{L_fh_{i-1}(x_0)}{b_{i-1} - h_{i-1}(x_0)}\right\}\label{control_gain},
\end{align}
which is strictly positive.
% \begin{align}
%     b_i \coloneqq c_{i-1}b_{i - 1}.
% \end{align}

The induced safe set for the target CBF is then:
\begin{align}
    \mathcal{C}_{h_n} \cap \mathcal{C}_{\hbar_n},\label{safe_set_induced}
\end{align} 
where $\mathcal{C}_{h_n} = \mathcal{C}_h \cap \mathcal{C}_2 \cap \cdots \cap \mathcal{C}_i$ and $\mathcal{C}_{\hbar_n} = \mathcal{C}_\hbar \cap \bar{\mathcal{C}}_2 \cap \cdots \cap \bar{\mathcal{C}}_i$, with $\mathcal{C}_i = \{x \in D \mid h_i(x) \geq 0\}$ and $\bar{\mathcal{C}}_i = \{x \in D \mid \hbar_i(x) \geq 0\}$, where $i \in \{2,\cdots,n\}$. 

It is important to note, while it may seem that the set $\mathcal{C}_{h_n}\cap\mathcal{C}_{\hbar_n}$ does not include all $x_0 \in \text{Int}(\mathcal{C}_h) \cap \text{Int}(\mathcal{C}_\hbar)$, the gain choice in \eqref{control_gain} adapts the sets $\mathcal{C}_i$ and $\bar{\mathcal{C}}_i$ based on the initial condition $x_0$. As a result, for all $x_0 \in \text{Int}(\mathcal{C}_h) \cap \text{Int}(\mathcal{C}_\hbar)$, the point $x_0$ is included in $\mathcal{C}_{h_n} \cap \mathcal{C}_{\hbar_n}$. This is a key difference from HOCBFs \cite{xiao_hocbf}.

Then, we see that if the set $\mathcal{C}_{h_n}\cap\mathcal{C}_{\hbar_n}$ is kept forward invariant, then by design $h_1(x)$ and $\hbar_1(x)$ remain positive for all $t \in [t_0,\infty)$. However, it must be verified that $h_n(x)$ and $\hbar_n(x)$ are valid CBFs at least in the set $\mathcal{C}_{h_n}\cap\mathcal{C}_{\hbar_n}$. That is, we must verify that 
\begin{align}
    L_gh_n(x) &= 0 \implies L_fh_n(x)  \geq -\alpha_n(h_n(x)),\label{cbfcond_hn}\\
    L_g\hbar_n(x) &= 0 \implies L_f\hbar_n(x) \geq -\overline{\alpha}_n(\hbar_n(x))\label{cbfcond_hbarn},
\end{align}
where $\alpha_n$ and $\overline{\alpha}_n$ are extended class $\mathcal{K}_\infty$ functions, for all $x \in \mathcal{C}_{h_n}\cap\mathcal{C}_{\hbar_n}$. Verifying their validity as CBFs is highly dependent on the system and constraints. Hence, we take it as the following assumption:
\begin{assumption}\label{assum_hn_cbf}
    $h_n(x)$ and $\hbar_n(x)$ are valid CBFs for system \eqref{affine} in the set $\mathcal{C}_{h_n}\cap\mathcal{C}_{\hbar_n}$.
\end{assumption}

We summarize this result in Theorem \ref{thrm_highrel_safe}.

\begin{theorem}\label{thrm_highrel_safe}  
For system \eqref{affine}, suppose $ h_1(x) $ and $ \hbar_1(x) $, with relative degree $ r = n > 1 $, are parallel and satisfy Assumption \ref{assum_backstepping}. Let $ h_n(x), \hbar_n(x) $ (as in \eqref{hi}--\eqref{hbari}) satisfy Assumption \ref{assum_hn_cbf}. Then, there exist extended class $ \mathcal{K}_\infty $ functions $ \alpha_n$ and $ \overline{\alpha}_n $ such that applying \eqref{control_law_rel1} with $ h(x) = h_n(x) $, $ \hbar(x) = \hbar_n(x) $ to \eqref{affine} guarantees $ h_1(x(t)), \hbar_1(x(t)) \geq 0 $ for all $ t \in [t_0, \infty) $. Namely, the system remains safe under all nominal control inputs.
\end{theorem}

\begin{proof}
    Since $h_1(x)$ and $\hbar_1(x)$ are parallel, it follows that $h_n(x)$ and $\hbar_n(x)$ are also parallel, as seen in equations \eqref{hi}--\eqref{hbari}. Also, by Assumption \ref{assum_backstepping} and the CBF backstepping design, $h_n(x_0),\hbar_n(x_0) >0$ is guaranteed. Next, since $h_n(x)$ and $\hbar_n(x)$ are valid CBFs, we invoke Theorem \ref{thrm_rel1_Safe} to guarantee that the set $ \mathcal{C}_{h_n} \cap \mathcal{C}_{\hbar_n} $ is forward invariant. 

    By the CBF backstepping design, the condition ${h_n(x(t)), \hbar_n(x(t)) \geq 0}$ for all ${t \in [t_0, \infty)}$ implies that $ \dot{h}_{n-1}(x(t)) \geq -c_{n-1} h_{n-1}(x(t)) $ and $ \dot{\hbar}_{n-1}(x(t)) \geq -c_{n-1} \hbar_{n-1}(x(t)) $ for all $ t \in [t_0, \infty) $. Since Assumption \ref{assum_backstepping} and the gain choices in \eqref{control_gain} guarantee that $ h_{n-1}(x_0), \hbar_{n-1}(x_0) > 0 $, we invoke Lemma \ref{lemma_forward_invariance} to conclude that $ h_{n-1}(x(t)), \hbar_{n-1}(x(t)) \geq 0 $ for all $ t \in [t_0, \infty) $. Iterating this argument, it follows that $ h_1(x(t)), \hbar_1(x(t)) \geq 0 $ for all $ t \in [t_0, \infty) $.
\end{proof}

\TODO{Compared to recent works such as \cite{taylor_safe_2022} and \cite{ong_rectified_2024}, our method offers advantages in both design flexibility and implementation. While \cite{taylor_safe_2022} can be used to address the parallel boundary issue (as was shown in \cite{cohen_backstep_rom_2024}), its safe set can be overly conservative due to parameter choices, which may require retuning to accommodate for certain initial conditions in the original safe set. In contrast, our method adaptively selects gains \eqref{control_gain} based on the initial condition, ensuring it always lies within the induced safe set.

Relative to the Rectified CBF (ReCBF) approach in \cite{ong_rectified_2024}, our framework is simpler to implement, particularly for high relative degree CBFs. The work in \cite{ong_rectified_2024} uses ReLU (or $\max$) functions to rectify high relative degree CBFs, but because the QP safety filters require differentiating the CBF, the approach relies on carefully chosen hyperparameters to make the ReLU expressions continuously differentiable. This becomes increasingly delicate as the relative degree increases, since the ReLU expressions may need to be continuously differentiable up to $n-1$ times, where $n$ is the original CBF's relative degree. Our method, by comparison, provides a more straightforward algorithm with tunable parameters that only need to satisfy the condition in \eqref{control_gain}.
}

\begin{figure*}[t!] 
\vspace{1em}
    \centering
    \begin{subfigure}[b]{0.58\textwidth}
        \centering
        \includegraphics[height=5.3cm]{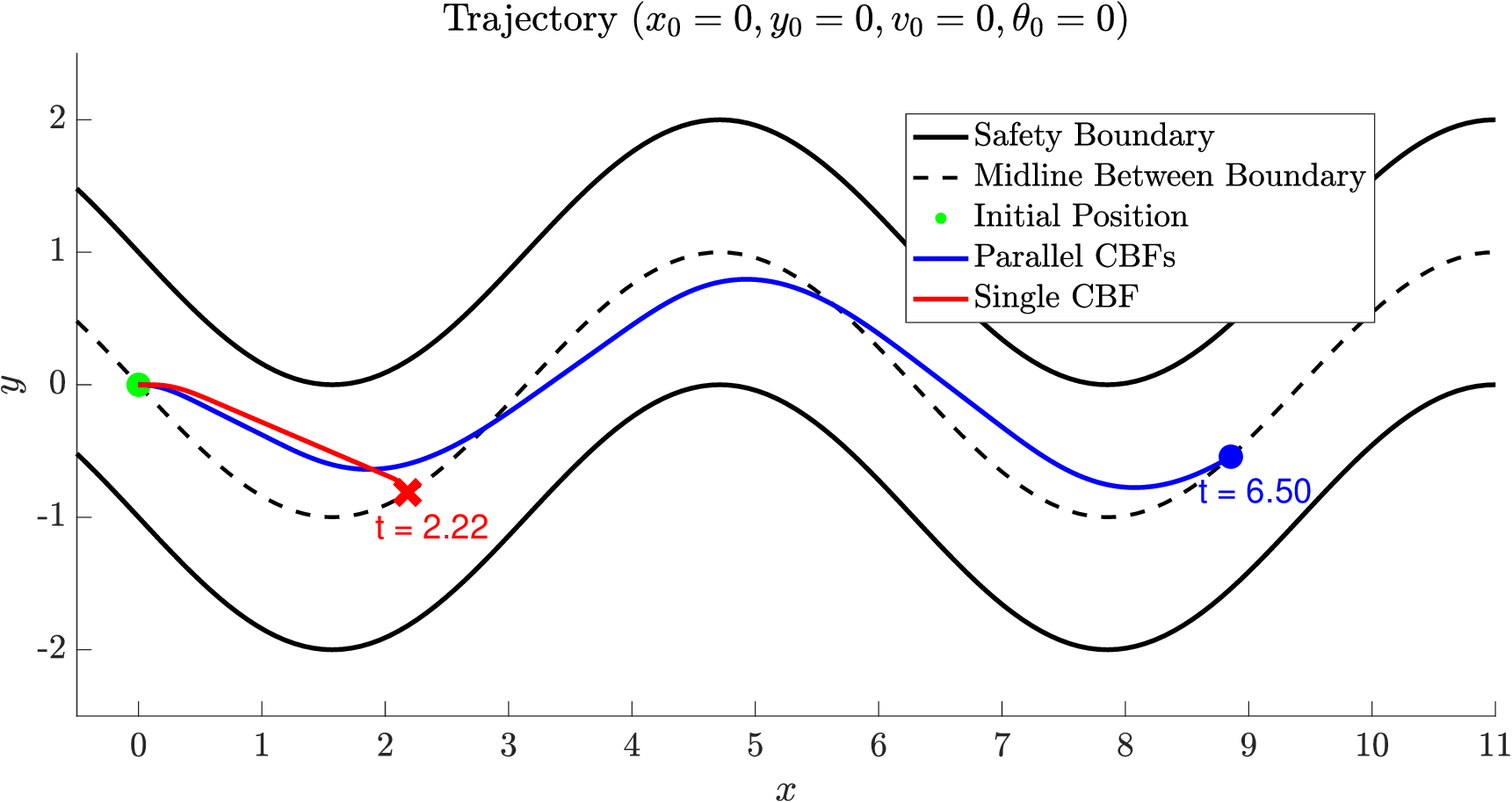}  % Adjust to your image path
        \caption{System Trajectory for $\mathbf{x}_0 = 0$.}
        \label{fig:fig1}
    \end{subfigure}
    \hfill
    \begin{subfigure}[b]{0.37\textwidth}
        \centering
        \includegraphics[height=5.3cm]{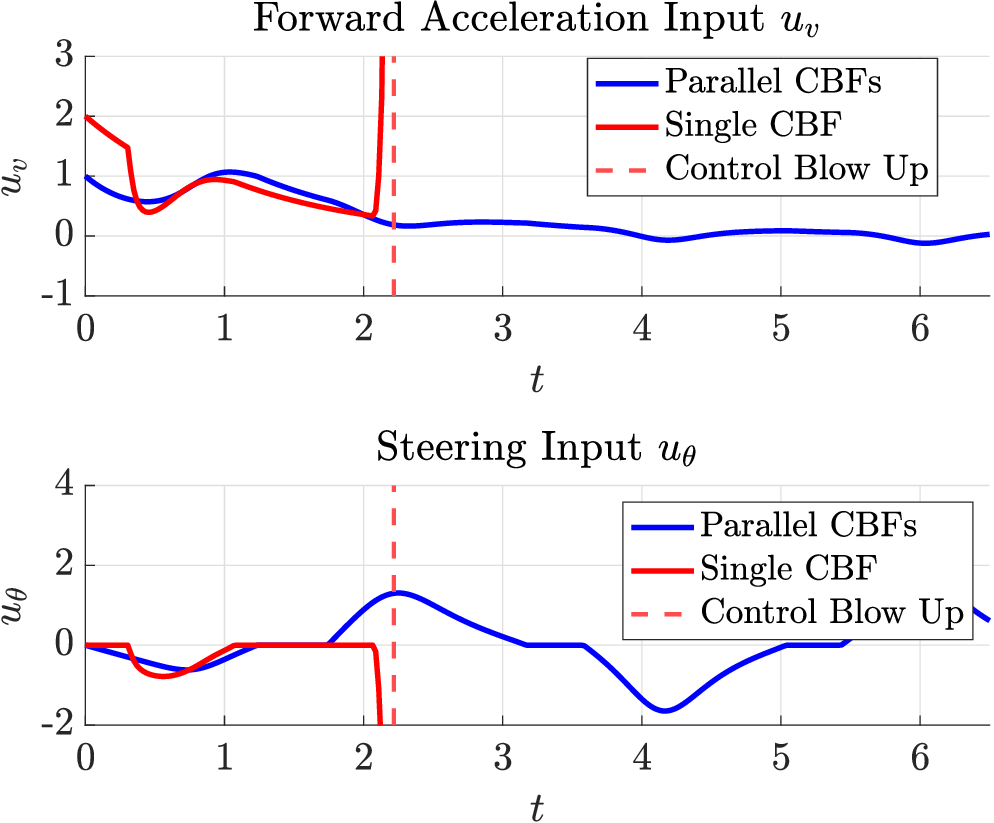}  % Adjust to your image path
        \caption{Control Action}
        \label{fig:fig2}
    \end{subfigure}
    % \caption{Simulation of system \eqref{uni_sys} using parallel CBFs $h_1(x,y) = \sin x + y + 1$ and $\hbar_1(x,y)= -\sin x - y + 1$ (blue), in comparison to a single CBF $h_s(x,y) = 1 - (\sin x +y)^2$ (red). The system was initialized at $\mathbf{x}_0 = 0$ with nominal controls $u_{0v} = 2 - v$, $u_{0\theta} = 0$, all relevant control gains were set to $1$ and all extended class $\mathcal{K}_\infty$ functions were defined to be the identity function $\alpha(s) = s$. (a) shows the trajectories; (b) shows control inputs. The single CBF controller blows-up near $\sin x + y = 0$ at $t \approx 2.22$ due to a vanishing gradient while the parallel CBF trajectories were well-defined throughout the time-interval.}
    \caption{\TODO{Simulation of system \eqref{uni_sys} using parallel CBFs $h_1(x,y) = \sin x + y + 1$ and $\hbar_1(x,y) = -\sin x - y + 1$ (blue), compared to a single CBF $h_s(x,y) = 1 - (\sin x + y)^2$ (red). The system was initialized at $\mathbf{x}_0 = 0$ with nominal controls $u_{0v} = 2 - v$, $u_{0\theta} = 0$; all control gains were set to $1$, and all extended class $\mathcal{K}_\infty$ functions were chosen as the identity functions $\alpha(s) = s$. (a) illustrates the system trajectories, and (b) depicts the corresponding control inputs. The vanishing gradient drives the single CBF controller to blow up near $\sin x + y = 0$ at $t \approx 2.22s$, while the parallel CBF controller remains well-defined throughout the simulation.}}

    \label{fig:uni_example}
\end{figure*}

\section{Examples and Simulations}

\subsection{Returning to the Motivational Problem}

We revisit the double integrator system $\dot{x}_1 = x_2, \quad \dot{x}_2 = u$, with the safe set $\mathcal{C} = \{x \in \mathbb{R}^2 \mid x_1 \in [-1,1]\}$. We define the CBFs $h(x) = 1 + x_1$ and $\hbar(x) = 1 - x_1$, corresponding to the safe sets $\mathcal{C}_h = \{x \in \mathbb{R}^2 \mid x_1 \geq -1\}$ and $\mathcal{C}_\hbar = \{x \in \mathbb{R}^2 \mid x_1 \leq 1\}$, whose intersection forms $\mathcal{C}$. Since $h(x) + \hbar(x) = 2$, they are parallel. Using backstepping, we assume $h(x_0), \hbar(x_0) > 0$ initially and choose:
\begin{align}
    c_1 = \bar{c}_1 > \max\left\{\frac{x_2(0)}{1 + x_1(0)}, \frac{x_2(0)}{1 - x_1(0)}\right\}.
\end{align}

The target CBFs are then defined as $ h_2(x) = c_1(1 + x_1) + x_2 $ and $ \hbar_2(x) = c_1(1 - x_1) - x_2 $, with $ L_g h_2(x) = 1 $ and $ L_g \hbar_2(x) = -1 $. Since $ L_g h(x) = -L_g \hbar(x) \neq 0 $ for all $ x \in \mathcal{C}_h\cap\mathcal{C}_\hbar $, they trivially satisfy \eqref{cbf_Lgh_cond} and are valid CBFs in $ \mathcal{C}_h\cap\mathcal{C}_\hbar $. Then, we invoke Theorem \ref{thrm_highrel_safe} to guarantee forward invariance of $ \mathcal{C} $.
% The two target CBFs are then defined as $ h_2(x) = c_1(1 + x_1) + x_2$ and $\hbar_2(x) = c_1(1 - x_1) - x_2$. Here, we observe that $L_g h_2(x) = 1$ and $ L_g \hbar_2(x) = -1$. Then, $h_2(x)$ and $\hbar_2(x)$ satisfy the condition in equation \eqref{cbf_Lgh_cond} trivially as $L_gh(x) = -L_gh(x) \neq 0$ for all $x \in \mathcal{C}_h\cap\mathcal{C}_\hbar$ making them valid CBFs in the set $\mathcal{C}_h\cap\mathcal{C}_\hbar$. Then, we invoke Theorem \ref{thrm_highrel_safe} to keep the set $\mathcal{C}$ forward invariant.

\subsection{Unicycle Between Parallel Boundaries}

Consider the kinematic unicycle system $\dot{x} = v \cos\theta, \dot{y} = v \sin\theta, \dot{\theta} = u_\theta$, where $[x,y]^\top \in \mathbb{R}^2$ are the positional states, $\theta \in \mathbb{R}$ is the heading angle, $v \in \mathbb{R}$ is the forward velocity input, and $u_\theta \in \mathbb{R}$ is the angular velocity input.

Suppose we are given two parallel positional constraints represented by candidate CBFs $h_1(x,y)$ and $\hbar_1(x,y)$, with safe sets $\mathcal{C}_{h}$ and $\mathcal{C}_{\hbar}$. Since both depend only on $x$ and $y$, they have a relative degree of 1 with respect to $v$ and 2 with respect to $u_\theta$, resulting in a mixed relative degree problem.

% Consider the kinematic unicycle system $\dot{x} = v \cos\theta, \; \dot{y} = v \sin\theta, \; \dot{\theta} = u_\theta$, where $[x,y]^\top \in \mathbb{R}^2$ are the positional states, $\theta \in \mathbb{R}$ is the heading angle, $v \in \mathbb{R}$ is the forward velocity input, and $u_\theta \in \mathbb{R}$ is the angular velocity input. 

% Suppose we are given two parallel positional constraints in the form of candidate CBFs $h_1(x,y)$ and $\hbar_1(x,y)$ with the associated safe sets $\mathcal{C}_{h}$ and $\mathcal{C}_{\hbar}$. However, since $h_1(x,y)$ and $\hbar_1(x,y)$ depend only on $x$ and $y$, they have a relative degree of 1 with respect to $v$ and a relative degree of 2 with respect to $u_\theta$, hence a mixed relative degree problem.

To first address this, we add an integrator $\dot{v} = u_v$ and assume control over the vehicle's forward acceleration similar to that in \cite{kim2024robustcontrolbarrierfunction}. That is,
\begin{align}\label{uni_sys}
    \dot{\mathbf{x}}
    =
    \begin{bmatrix}
        \dot{x}\\
        \dot{y}\\
        \dot{v}\\
        \dot{\theta}
    \end{bmatrix}
    =
    \underbrace{
    \begin{bmatrix}
        v\cos\theta\\
        v\sin\theta\\
        0\\
        0
    \end{bmatrix}}_{f(\mathbf{x})}
    +
    \underbrace{
    \begin{bmatrix}
        0\\
        0\\
        1\\
        0
    \end{bmatrix}}_{g_v(\mathbf{x})}u_v
    +
    \underbrace{
    \begin{bmatrix}
        0\\
        0\\
        0\\
        1
    \end{bmatrix}}_{g_\theta(\mathbf{x})}u_\theta.
\end{align}

This essentially modifies the unicycle model to a simplified bicycle model \cite{rahman_driver_2021} and results in a uniform relative degree $2$ problem. We now proceed with the backstepping procedure with the assumption that $h_1(x_0,y_0),\hbar_1(x_0,y_0)>0$.

Since $h_1(x,y)$ and $\hbar_1(x,y)$ are parallel, without loss of generality, we compute the backstepping terms with respect to $h_1(x,y)$. 
% Taking the first derivative:
% \begin{align}
%     \dot{h}_1(\mathbf{x} )&= -\dot{\hbar}_1(\mathbf{x}),\\
%     &=L_fh_1(\mathbf{x})=\nabla h_1(\mathbf{x})^\top f(\mathbf{x}),\\
%     % &= \frac{\partial h_1}{\partial x}\left(v\cos\theta\right) + \frac{\partial h_1}{\partial y}\left(v\sin\theta\right)\\
%     &= v\left(\nabla_{x,y} h_1(x,y)^\top \cdot \begin{bmatrix}
%         \cos\theta\\
%         \sin\theta
%     \end{bmatrix}\right).
% \end{align}
% % \begin{align}
% %     c_1 = \bar{c}_1 > \max\left\{ -\frac{L_fh_1(\mathbf{x}_0)}{h_1(x_0,y_0)}, \frac{L_fh_1(\mathbf{x}_0)}{b_1 - h_1(x_0,y_0)}\right\}\label{uni_gain},
% % \end{align}
By choosing the appropriate gains $c_1 = \overline{c}_1 > 0$ as in \eqref{control_gain}, we get the following target CBF through the backstepping procedure:
\begin{align}
    h_2(\mathbf{x}) = b_2 - \hbar_2(\mathbf{x}) = c_1h_1(x,y) + L_fh_1(\mathbf{x})\label{h2},
\end{align}
with a corresponding safe set $\mathcal{C}_{h_2}\cap\mathcal{C}_{\hbar_2}$ as in \eqref{safe_set_induced}. Then, defining $L_gh_2(\mathbf{x}) = [L_{g_v}h_2(\mathbf{x}),L_{g_\theta}h_2(\mathbf{x})]$ and $u = [u_v,u_\theta]^\top$, the time derivative is:
\begin{align}
    \dot{h}_2(\mathbf{x}) 
    % &= -\dot{\hbar}_2(\mathbf{x})\\
    % &=c_1L_fh_1(\mathbf{x}) + L_f^2h_1(\mathbf{x}) + L_gL_fh_1(\mathbf{x}),\\
    &= L_gh_2(\mathbf{x)}u + c_1v\left(\nabla_{x,y} h_1(x,y)^\top \cdot\begin{bmatrix}
        \cos\theta\\
        \sin\theta
    \end{bmatrix}\right)\nonumber\\
    &+  v^2\left(
    \begin{bmatrix}
        \cos\theta\\
        \sin\theta
    \end{bmatrix}^\top
    \cdot\nabla^2_{x,y} h_1(x,y) \cdot 
    \begin{bmatrix}
        \cos\theta\\
        \sin\theta
    \end{bmatrix}\right)\label{h2_dot_uni},
\end{align}
where 
\begin{align}
    L_{g_v}h_2(\mathbf{x}) &= \nabla_{x,y} h_1(x,y)^\top \cdot\begin{bmatrix}
        \cos\theta\\
        \sin\theta
    \end{bmatrix}\label{Lg_v_h},\\
    L_{g_\theta}h_2(\mathbf{x}) &= v\left(\nabla_{x,y} h_1(x,y)^\top \cdot \begin{bmatrix}
        -\sin\theta\\
        \cos\theta
    \end{bmatrix}\right)\label{Lg_theta_h}.
\end{align}
% $L_{g_v}h_2(\mathbf{x}) = \nabla_{x,y} h_1(x,y)^\top \cdot [\cos\theta, \sin\theta]^\top$ and $L_{g_\theta}h_2(\mathbf{x}) = v(\nabla_{x,y} h_1(x,y)^\top \cdot [-\sin\theta, \cos\theta]^\top$.
Intuitively, if $ \nabla_{x,y} h_1(x,y) \neq 0 $, the control $ u_v $ loses agency over safety ($ L_{g_v} h_2(\mathbf{x}) = 0 $) only when the heading angle is perpendicular to the gradient, i.e., parallel to the safety boundary. Similarly, $ u_\theta $ loses agency ($ L_{g_\theta} h_2(\mathbf{x}) = 0 $) when the heading angle is parallel to the gradient, i.e., perpendicular to the boundary, or if $ v = 0 $. Thus, $ L_g h_2(\mathbf{x}) = 0 $ only if the heading is parallel to the boundary and $ v = 0 $, forming the basis of the following claim.
% Intuitively, we see that if $\nabla_{x,y}h_1(x,y) \neq 0$, then the forward acceleration control $u_v$ only loses agency over safety (i.e., $L_{g_v}h_2(\mathbf{x}) = 0$) when the heading angle is perpendicular to the gradient, hence parallel to the safety boundary. Similarly, the steering control $u_\theta$ only loses agency over safety (i.e., $L_{g_\theta}h_2(\mathbf{x}) = 0$) if the heading angle is parallel to the gradient, hence perpendicular to the safety boundary, or if the forward velocity $v = 0$. Thus, if $\nabla_{x,y}h_1(x,y) \neq 0$ the only way $L_gh_2(\mathbf{x}) = 0$ is for the heading angle to be parallel to the safety boundary and $v = 0$. This intuition forms the basis of the following claim.

\begin{theorem}\label{thrm_uni_parallel_cbf}
    For system \eqref{uni_sys}, given parallel positional constraints $h_1(x,y)$ and $\hbar_1(x,y)$, if $h_1(x_0,y_0),\hbar_1(x_0,y_0) > 0$, then $\nabla_{x,y} h_1(x,y) \neq 0$ for all $x,y \in \mathcal{C}_{h}\cap\mathcal{C}_{\hbar}$ is a sufficient condition for $h_2(\mathbf{x})$ and $\hbar_2(\mathbf{x})$ (as in \eqref{h2}) to be valid CBFs in the set $\mathcal{C}_{h_2}\cap\mathcal{C}_{\hbar_2}$.
\end{theorem}
\begin{proof}
    Suppose $\nabla_{x,y}h_1(x,y) \neq 0$ for all $x,y \in \mathcal{C}_{h}\cap\mathcal{C}_{\hbar}$. To establish that $h_2(\mathbf{x})$ and $\hbar_2(\mathbf{x})$ are CBFs in $\mathcal{C}_{h_2}\cap\mathcal{C}_{\hbar_2}$, it suffices to verify that $L_g h_2(\mathbf{x}) = 0$ implies $L_f h_2(\mathbf{x}) \geq -\alpha_2(h_2(\mathbf{x}))$ and $L_g \hbar_2(\mathbf{x}) = 0$ implies $L_f \hbar_2(\mathbf{x}) \geq -\bar{\alpha}_2(\hbar_2(\mathbf{x}))$ for all $\mathbf{x} \in \mathcal{C}_{h_2}\cap\mathcal{C}_{\hbar_2}$ as in \eqref{cbfcond_hn}--\eqref{cbfcond_hbarn}. Noting that $[\cos\theta, \sin\theta]^\top$ and $[-\sin\theta, \cos\theta]^\top$ are orthogonal, equations $L_{g_v}h_2$ and $L_{g_\theta}h_2$ vanish simultaneously only when $v = 0$ and  
    \begin{align}
        \nabla_{x,y} h_1(x,y)^\top \cdot\begin{bmatrix}
            \cos\theta \\ \sin\theta
        \end{bmatrix} = 0.
    \end{align}  
    Under this condition, \eqref{h2_dot_uni} evaluates to zero, implying $L_g h_2(\mathbf{x}) = 0 \implies L_f h_2(\mathbf{x}) = 0$ and $L_g \hbar_2(\mathbf{x}) = 0 \implies L_f \hbar_2(\mathbf{x}) = 0$. Since $0 \geq -\alpha_2(h_2(\mathbf{x}))$ and $0 \geq -\bar{\alpha}_2(\hbar_2(\mathbf{x}))$ holds for all $\mathbf{x} \in \mathcal{C}_{h_2}\cap\mathcal{C}_{\hbar_2}$,  $h_2(\mathbf{x})$ and $\hbar_2(\mathbf{x})$ are valid CBFs on $\mathcal{C}_{h_2}\cap\mathcal{C}_{\hbar_2}$.
\end{proof}

% The immediate Corollary to this is for any singular CBF $h_1(x,y)$ will emite a valid CBF bla The immediate Corollary to this is for any singular CBF $h_1(x,y)$ will emite a valid CBF blaThe immediate Corollary to this is for any singular CBF $h_1(x,y)$ will emite a valid CBF bla

% Then, once the CBFs have been validated, we simply invoke Theorem \ref{thrm_highrel_safe} to keep the system \eqref{uni_sys} safe with respect to the constraints $h_1(x,y)$ and $\hbar_1(x,y)$.

\TODO{As an example, consider the positional constraints 
\begin{align}
    h_1(x,y) &= \sin x + y + 1,\\
    \hbar_1(x,y) &= -\sin x - y + 1.
\end{align}
% $h_1(x,y) = \sin x + y + 1$ and $\hbar_1(x,y) = -\sin x - y + 1$. 
These constraints are parallel since $h_1(x,y) + \hbar_1(x,y) = 2$, and their gradients are nonzero as $\nabla_y h_1(\mathbf{x}) = -\nabla_y \hbar_1(\mathbf{x}) = 1$, satisfying the sufficient condition in Theorem~\ref{thrm_uni_parallel_cbf}.

We compare this to the single CBF 
% \begin{align}
%     h_s(x,y) = 1 - (\sin x + y)^2,
% \end{align}
$h_s(x,y) = 1 - (\sin x + y)^2$, 
which defines the same safe set as the intersection of the two parallel constraints. However, when using the CBF backstepping (or HOCBF) method, this formulation loses control authority along the midline $\sin x + y = 0$, where $\nabla h_s(x,y) = 0$. On this line, the CBF becomes ill-defined, and its validity depends on the system’s forward velocity and heading angle.

Figure~\ref{fig:uni_example} shows a simulation comparing both approaches. Both systems were initialized at the origin (i.e., $\mathbf{x}_0 = 0$), which satisfies $h_1(x_0,y_0), \hbar_1(x_0,y_0), h_s(x_0,y_0) > 0$. All gain choices were set to $1$ (satisfying the backstepping design constraint), and all extended class $\mathcal{K}_\infty$ functions were chosen as the identity function (i.e., $\alpha(s) = s$). The nominal control laws were $u_{0v} = 2 - v$ and $u_{0\theta} = 0$.

The single CBF-based safety filter (red) initially produced a valid trajectory, but as it re-encountered the midline, the vanishing gradient caused the CBF condition to fail, leading to a controller blow up at $t \approx 2.22s$. In contrast, the parallel CBF-based safety filter (blue) remained well-defined, as the gradients never vanished in the safe set, and successfully kept the system safe with respect to the parallel boundaries.
}

\section{Conclusion and Future Directions}  
 \TODO{In this paper, we addressed the challenge of valid CBFs for systems with parallel safety boundaries, a problem where relying on a single CBF  is often invalid at the midway point between the boundaries. We identified that the core limitation arises from the vanishing gradient within the interior of the safe set and proposed an alternative approach that constructs separate CBFs for each boundary, effectively addressing the issue. This was validated through a simulation on a unicycle model.}
    
\TODO{Our method was based on unbounded control and CBFs with uniform relative degrees. While it does not directly address input constraints or mixed relative degree CBFs without modifications, extending the framework to handle such cases is a promising direction for future research. Broadening its applicability to unmodified systems like the unicycle is especially important in cases where structural modification is infeasible.}
%%%%%%%%%%%%%%%%%%%%%%%%%%%%%%%%%%%%%%%%%%%%%%%%%%%%%%%%
% \balance
%%%%%%%%%%%%%%%%%%%%%%%%%%%%%%%%%%%%%%%%%%%%%%%%%%%%%
% \newpage
\bibliographystyle{ieeetr}
\bibliography{root}
%%%%%%%%%%%%%%%%%%%%%%%%%%%%%%%%%%%%%%%%%%%%%%%%%%%%%%%%
\end{document}